\documentclass[11pt]{amsart} 
\IfFileExists{srcltx.sty}{\usepackage[active]{srcltx}}{}

\usepackage{dsfont}

\usepackage{amscd}
\usepackage{amsfonts,amssymb,latexsym}
\usepackage[all]{xy}

\xyoption{arc}

\CompileMatrices

\newcommand{\SE}{{\mathcal{E}}}

\newcommand{\SH}{{\mathcal{H}}}

\newcommand{\SM}{{\mathcal{M}}}

\newcommand{\SO}{{\mathcal{O}}}

\newcommand{\SQ}{{\mathcal{Q}}}

\newcommand{\cS}{{\mathcal{S}}}
\newcommand{\ST}{{\mathcal{T}}}

\newcommand{\SW}{{\mathcal{W}}}

\newcommand{\SY}{{\mathcal{Y}}}

\newcommand{\PP}{\mathbb{P}}
\newcommand{\ZZ}{\mathbb{Z}}
\newcommand{\NN}{\mathbb{N}}
\newcommand{\CC}{\mathbb{C}}

\newcommand{\codim}{\operatorname{codim}}

\newcommand{\Pic}{\operatorname{Pic}}

\newcommand{\inj}{\hookrightarrow}
\newcommand{\too}{\longrightarrow}
\newcommand{\rk}{\operatorname{rk}}

\newcommand{\wt}{\widetilde}

\newcommand{\GL}{{\operatorname{GL}}}
\newcommand{\GO}{\operatorname{GO}}
\newcommand{\orth}{\operatorname{O}}
\newcommand{\PGL}{{\operatorname{PGL}}}

\newcommand{\fq}{\mathfrak{q}}

\newcommand{\fz}{\mathfrak{z}}

\newcommand{\fgo}{\mathfrak{go}}
\newcommand{\fso}{\mathfrak{so}}
\newcommand{\fgl}{\mathfrak{gl}}

\newcommand{\udots}{\mathinner{\mskip1mu\raise1pt\vbox{\kern7pt\hbox{.}}
\mskip2mu\raise4pt\hbox{.}\mskip2mu\raise7pt\hbox{.}\mskip1mu}}

\newtheorem{proposition}{Proposition}[section]
\newtheorem{theorem}[proposition]{Theorem}
\newtheorem{definition}[proposition]{Definition}

\newtheorem{lemma}[proposition]{Lemma}

\newtheorem{corollary}[proposition]{Corollary}

\numberwithin{equation}{section}

\begin{document}

\title[Hecke transformation for orthogonal bundles]{Hecke
transformation for orthogonal bundles and stability of Picard bundles}

\author[I. Biswas]{Indranil Biswas}

\address{School of Mathematics, Tata Institute of Fundamental
  Research, Homi Bhabha Road, Mumbai 400005, India}

\email{indranil@math.tifr.res.in}

\author[T.L. G\'omez]{Tom\'as L. G\'omez}

\address{Instituto de Ciencias Matem\'aticas (CSIC-UAM-UC3M-UCM), 
Nicolas Cabrera 13, Campus Cantoblanco UAM, 28049 Madrid;
and
Facultad de Ciencias Matem\'aticas, 
Universidad Complutense de Madrid, 28040 Madrid, Spain}

\email{tomas.gomez@icmat.es}

\date{}



\thanks{Research supported by the Spanish
Ministerio de Educaci\'on y Ciencia [MTM2007-63582]}

\begin{abstract}
We define Hecke transformation for orthogonal bundles over a
compact Riemann surface. Using the cycles on a moduli space of
orthogonal bundles given by Hecke transformations, we prove that
the projectivized Picard bundle on the moduli space is stable.
\end{abstract}

\maketitle

\section{Introduction}

Given a holomorphic vector bundle $F$ over a compact Riemann surface
$X$, and a subspace $S_x\, \subset\, F_{x}$ in the fiber over a
point $x$,
the Hecke transformation produces a new vector bundle $E$ on $X$
\cite{Ty}, \cite{NR1}.
The vector bundle $E$ is the kernel of the natural quotient map
$F\,\too\, F_{x}/S_x$. Hecke transformation is a very useful tool
to study the moduli space. For instance,
they are used in computing cohomologies of coherent sheaves on a
moduli space of vector bundles \cite{NR1}. They are also used in proving
stability of various naturally associated bundles on a moduli space
\cite{BBGN}.

When $S_x$ varies among all subspaces of $F_x$ (the fiber of $F$ at
$x$), with
$x$ fixed, we get a family
of vector bundles. Under suitable conditions for $F$, these
Hecke transforms are stable vector bundles, so we obtain a
morphism from the Grassmannian associated
to $F_x$ to the moduli space of vector bundles.
The image of this morphism is called a Hecke cycle.

An orthogonal bundle is a vector bundle $F$ together with a
homomorphism $\psi:F\otimes F\too M$, where $M$ is a line bundle,
such that $\psi$ is symmetric and non-degenerate at every fiber.
Equivalently, an orthogonal bundle can be thought of as a principal
$\GO(r,\CC)$-bundle.
Our aim here is systematically to construct Hecke transformations
of orthogonal bundles. If $F$ is an orthogonal bundle over $X$
of rank $2n$, and $S_x\, \subset\, F_{x}$ is an isotropic subspace
of dimension $n$, then the vector bundle $E\too X$
defined  by the kernel of the homomorphism
$F\,\too\, F_{x}/S_x$ has an induced orthogonal structure. If the
orthogonal form on $F$ takes values in a line bundle $M$, then the
orthogonal form on $E$ takes values in
$M\otimes {\mathcal O}_X(-x)$. Summing up, we start with a principal 
$\GO(2n,\CC)$--bundle $F$ and a Lagrangian subspace of $F_x$,
and we obtain another $\GO(2n,\CC)$--bundle. If we let 
$S_x$ vary, we will obtain a family of $\GO(2n,\CC)$--bundles.
Under suitable conditions on $(F,\psi)$, all these bundles are
stable, hence we obtain a morphism to the moduli space of stable
orthogonal bundles, whose image is called a \emph{Hecke cycle}.

For odd ranks, we consider vector bundles $F\too X$ of rank
$2n+1$ equipped with a symmetric bilinear form $\psi$ which is 
nondegenerate on $X\setminus \{x\}$, and the annihilator $l_x$
of $F_{x}$ is of dimension one. For any isotropic subspace
$\wt S_x\, \subset\, F_{x}$ of dimension $n+1$ (or, equivalently,
for any isotropic subspace $S_x\, \subset \, F_x/l_x$ of 
dimension $n$), we construct an
orthogonal bundle.
As in the even case,
we can define a morphism to the moduli space whose image is called
a \emph{Hecke cycle}.
Note that in the odd case, in order to obtain a principal
bundle, we start with an object which is not a principal 
bundle (the bilinear form 
on the fiber over $x$ is degenerate). From
this point of view, a Hecke transformation, rather
than a transformation between principal bundles, is better
understood as a technique for producing interesting cycles in the 
moduli space.

As an application, we prove that the
projectivized Picard bundle (see Section \ref{picard}
for the definition) on a moduli space
of orthogonal bundles is stable (Theorem \ref{thm.}).

In \cite{BG} we have considered symplectic Hecke transformations.
At the end of this article we comment on the differences
between the symplectic and orthogonal case.

\section{Preliminaries}

We fix a nondegenerate symmetric bilinear form $B$ on $\CC^r$,
$r\,\geq\, 3$.
The symmetric matrix representing $B$ will also be denoted
by $B$. Define the \emph{general orthogonal} group $\GO(r,\CC)$ to 
be the group of all conformally orthogonal
transformations, meaning
\begin{equation}
  \label{eq:go}
  \GO(r,\CC)=\{A\in \GL(r,\CC) \, : \, A^tBA=cB 
\;\text{for some}\; c\in \CC^*\}\, .
\end{equation}
This group is an extension of $\CC^*$ by the orthogonal
group $\orth(r,\CC)$
\begin{equation}
  \label{eq:goext}
1 \too \orth(r,\CC) \too \GO(r,\CC) \stackrel{p}{\too} \CC^* \too 1\, ,
\end{equation}
where $p(A)$ is the constant $c$ in \eqref{eq:go}. It follows that
$$
\det(A)^2=p(A)^r\, .
$$

Let $X$ be a compact connected Riemann surface.
An \emph{orthogonal bundle} on $X$ is a pair
of the form $(E,\varphi)$, where 
\begin{itemize}
\item $E\longrightarrow X$ is a holomorphic vector bundle, and

\item $\varphi$ is a symmetric and nondegenerate 
holomorphic homomorphism
$$
\varphi:E\otimes E \too L\, ,
$$
where $L\longrightarrow X$ is a holomorphic line bundle.
\end{itemize}
The homomorphism $\varphi$ induces an isomorphism 
$E\too E^\vee\otimes L$, and this in turn produces an isomorphism
\begin{equation}\label{e1}
\det(E)^2\, \stackrel{\sim}{\longrightarrow}\, L^r \; .
\end{equation}
An \emph{isomorphism} of orthogonal bundles
$$
(E,\varphi) \too (E',\varphi')
$$
is a pair of holomorphic
isomorphisms $(\alpha:E\stackrel{\sim}{\too} E',\, \beta:L
\stackrel{\sim}{\too} L')$
such that the following diagram is commutative
$$
\xymatrix{
{E\otimes E} \ar[r]^-{\varphi} \ar[d]_{\alpha\otimes\alpha} 
& {L}\ar[d]^{\beta}\\
{E'\otimes E'} \ar[r]^-{\varphi'} & {L'}
}
$$
There is a canonical bijection between the isomorphism classes of 
principal $\GO(r,\CC)$--bundles and orthogonal bundles of rank $r$.

Let $(E,\varphi:E\otimes E\too L)$ be an orthogonal bundle.
If $F\subset E$ is a subsheaf, we define $F^\perp$ to be the 
kernel of the composition
$$
E \stackrel{\varphi}\too E^\vee\otimes L \too 
F^\vee \otimes L \; .
$$
In other words, $F^\vee$ is the annihilator of $F$.

\begin{lemma}
\label{short}
Let $(E,\varphi:E\otimes E \too L)$ be an orthogonal bundle
on $X$, and let $F\subset E$ be a subbundle.
\begin{enumerate}

\item \label{short1}
There is a short exact sequence on $X$
$$
0 \too F^\perp \too E\,\stackrel{\varphi}{\cong}\, E^\vee\otimes L \too 
F^\vee\otimes L \too 0\, ,
$$
hence $F^\perp$ is also a subbundle, so $\rk F^\perp =\rk E -\rk F $, 
and $\deg F^\perp =\deg F +l(\frac{r}{2}-\rk  F )$, where $l$ is
the degree of $L$.

\item \label{short3}
There is a short exact sequence
\begin{equation}
\label{cap}
0 \too F\cap F^\perp \too F\oplus F^\perp \too F+F^\perp \too 0\, .
\end{equation}
\item \label{short4}
We have an inclusion $F+F^\perp \subset (F\bigcap F^\perp)^\perp$, in 
particular,
$\rk (F+F^\perp) \leq \rk (F\bigcap F^\perp)^\perp$.
\end{enumerate}
\end{lemma}

\begin{proof}
To prove (\ref{short1}), note that from \eqref{e1}
it follows that $\deg E \,=\, rl/2$. Also,
$F^\vee$ is a quotient bundle of $E^\vee$ because
$F$ is a subbundle of $E$. Now (\ref{short1})
follows from the definition of $F^\perp$.

Part (\ref{short3}) is easy to check.

For (\ref{short4}), note that if $F_1$ is a subsheaf of $F_2$,
then there is a natural inclusion $F_2^\perp\,\hookrightarrow\, 
F_1^\perp$. Since $F\bigcap F^\perp\subset F$ and
$F\bigcap F^\perp\subset F^\perp$, we have
$F^\perp \,\subset\, (F^\perp\bigcap F)^\perp$, and
$F\,=\, (F^\perp)^\perp \,\subset\, (F^\perp\bigcap F)^\perp$.
Hence $F+F^\perp\,\subset\, (F^\perp\bigcap F)^\perp$.
\end{proof}

A principal $\GO(r,\CC)$--bundle over a smooth complex
projective curve $X$ is
called \emph{stable} (respectively, 
\emph{semistable}) if for every reduction $\sigma\, :\, X
\,\longrightarrow\, P/Q$ to a (proper)
maximal parabolic subgroup $Q\subset \GO(r,\CC)$,
$$
\deg(\sigma^* T_{\rm rel})\, >\, 0
~\,~\,\text{~(respectively,~}\, \deg(\sigma^* T_{\rm rel})\,
\geq \, 0{\rm )}\, ,
$$
where $T_{\rm rel}\too P/Q$ is the relative tangent bundle
for the projection $P/Q\,\longrightarrow\, X$
(see \cite[page 129, Definition 1.1]{Ra1}
and \cite[page 131, Lemma 2.1]{Ra1}).
In terms of orthogonal bundles, this condition is equivalent to the 
condition
that for all isotropic subbundles $0\not= E'\subset E$, 
$$
\frac{\deg (E') }{\rk(E')} \, <\,  \frac{\deg (E)}{\rk(E)}
~\,~\,\text{~(respectively,~}\, \frac{\deg (E')}{\rk(E')} \, \leq\,
\frac{\deg (E)}{\rk(E)}{\rm )}\, ;
$$
we recall that $E'$ is
called an \textit{isotropic} subbundle of $E$ if the restriction of
$\varphi$ to $E'\otimes E'\, \subset\, E\otimes E$ is identically zero.

See \cite{Ra2,Ra3} for the construction of
moduli spaces of semistable principal $\GO(r,\CC)$--bundles.
We denote by $\SM_L$ the moduli space of stable orthogonal bundles with
values in the line bundle $L$.

\begin{lemma}
\label{Esemistable}
If $(E,\varphi)$ is a semistable orthogonal bundle on $X$, then
the underlying vector bundle $E$ is semistable.
\end{lemma}

\begin{proof}
The natural inclusion of $\GO(r,\CC)$ into $\text{GL}(r,\CC)$
takes the center of $\GO(r,\CC)$ into the center of
$\text{GL}(r,\CC)$. Therefore, the lemma follows from
\cite[p. 285, Theorem 3.18]{RR}.
\end{proof}

\section{Hecke transformation for orthogonal bundles}

Let $X$ be a compact connected Riemann surface of genus
$g\,=\,g(X)\,\geq\, 2$. Fix a point $x\,\in\, X$.

\begin{proposition}
\label{prop:even1}
  Let $(F,\psi: F\otimes F\too M)$ be an orthogonal bundle
over $X$ of rank $r=2n$. Let $S_x\subset F_x$
be an isotropic subspace of dimension $n$. Set $Q_x=F_x/S_x$,
and let
\begin{equation}
  \label{eq:even1}
  0 \too E  \too F \too Q_x \too 0
\end{equation}
be the short exact sequence,
where $F\too F_x\too Q_x$ is the natural projection. Then
$E$ inherits a natural orthogonal structure
$\varphi:E\otimes E \too L$, where $L=M(-x):=
M\otimes_{{\mathcal O}_X} {\mathcal O}_X(-x)$.
\end{proposition}

\begin{proof}
Choose a local \'etale trivialization of $(F,\psi)$ around $x\in X$
such that $\psi$ is of the form
$$
\psi=
\left(
  \begin{matrix}
    0 & & 1\\
      & \udots & \\
   1 &  & 0 \\
  \end{matrix}
\right)
$$
(so $\psi_{i,j}=0$ if $i+j\not= 2n$, and $\psi_{i,2n-i}=1$)
and $S_x$ is defined by the first $n$ vectors in the basis.
The homomorphism $E\too F$ is then locally defined by
the matrix
$$
\left(
  \begin{matrix}
   \mathds{1}_n  & \\
      & t\mathds{1}_n \\
  \end{matrix}
\right)
$$
(i.e., the diagonal matrix with the first $n$ entries equal to
$1$ and the last $n$ entries equal to $t$), 
where $t$ is a local parameter at $x\in X$. Therefore,
the composition $E\otimes E \too F\otimes F\too M$
is 
$$
\left(
  \begin{matrix}
    0 & & t\\
      & \udots & \\
   t &  & 0 \\
  \end{matrix}
\right)
$$
so it vanishes at $x\in X$. Therefore, the homomorphism
$E\otimes E\too M$ factors through
$L=M(-x)$, and then $\varphi:E\otimes E\too L$ is of the form
$$
\left(
  \begin{matrix}
    0 & & 1\\
      & \udots & \\
   1 &  & 0 \\
  \end{matrix}
\right)\, .
$$
This completes the proof.
\end{proof}

Let $(E,\varphi:E\otimes E\too L)$ be an orthogonal bundle
over $X$. Let $\widehat{\varphi} : E\longrightarrow\,
E^\vee \otimes L$ be the isomorphism given by $\varphi$.
Define
\begin{equation}\label{oi}
\varphi^{-1}\, :=\, (\widehat{\varphi}^*)^{-1}\, ,
\end{equation}
which produces a homomorphism
${\varphi}^{-1} :E^\vee \otimes E^\vee\too L^\vee$.
Note that
$(E^\vee, {\varphi}^{-1} :E^\vee \otimes E^\vee\too L^\vee)$
is an orthogonal bundle.

\begin{proposition}
\label{prop:even2}
  Let $(E,\varphi:E\otimes E\too L)$ be an orthogonal bundle
over $X$ of rank $r=2n$. Let $W$ be 
an isotropic subspace of dimension $n$ of $E_x^\vee$.
Let $F^\vee$ be defined by the following short exact
sequence
\begin{equation}
  \label{eq:even2}
0\,\too\, F^\vee \,\too\,  E^\vee \,\too\, E_x^\vee/W
\,\too\, 0\, .
\end{equation}
Then the orthogonal form ${\varphi}^{-1}$ on $E^\vee$ (see
\eqref{oi}) restricts to an orthogonal form
$$
{\varphi}^{-1} \, :\, F^\vee\otimes F^\vee\,\too\,
L^\vee(-x)
$$
on $F^\vee$.
\end{proposition}

\begin{proof}
The proposition follows by applying Proposition
\ref{prop:even1} to the orthogonal bundle 
$(E^\vee,\varphi^* :E^\vee \otimes E^\vee\too L^\vee)$
and the subspace $W\, \subset\, E_x^\vee$.
\end{proof}

Using \eqref{oi}, the orthogonal structure on $F^\vee$
in Proposition \ref{prop:even2} produces an orthogonal
structure
$$
F\otimes F\,\too\, L(x)
$$
on $F$.

\begin{proposition}
\label{prop:odd1}
Let $F\too X$ be a vector bundle of rank $r=2n+1$
equipped with a symmetric bilinear form
$$
\psi:F\otimes F\too M
$$
which induces a short exact sequence
\begin{equation}\label{exs}
0 \too F \too F^\vee\otimes M \too \CC_{x}\too 0\, ,
\end{equation}
where $\CC_x$ is the skyscraper sheaf of length one
supported over the point $x$.
In other words, $\psi$ 
is non-degenerate everywhere except at $x$, and
in an \'etale neighborhood of $x$, it is of the
form
\begin{equation}
\label{eq:psi}
\left(
  \begin{matrix}
    0 & & \mathds{1}_n\\
      & t & \\
    \mathds{1}_n & & 0\\
  \end{matrix}
\right)
\end{equation}
(i.e., the $(n+1,n+1)$-th entry is $t$ and any other $(i,j)$-th
entry is $\delta_{|i-j|, n+1}$),
where $t$ is a local parameter at $x\in X$. Let
\begin{equation}
\label{rest}
0\too l_x \too  F_x \too F_x^\vee\otimes M_x \too \CC_{x}\too 0
\end{equation}
be the exact sequence obtained by restricting the above
short exact sequence to the point $x$.
Let $S_x$ be an isotropic subspace of dimension $n$
of $F_x/l_x$.
Define $Q_x:=(F_x/l_x)/S_x$, and consider the
short exact sequence
$$
0 \too E \too  F \too Q_x \too 0\, .
$$
Then $\psi$ induces an orthogonal structure on $E$
$$
\varphi:E\otimes E \too L:=M(-x)\, .
$$
\end{proposition}

\begin{proof}
  Choose a local trivialization such that $\psi$ is of the
form in (\ref{eq:psi}), and furthermore,
the homomorphism $\beta:E\too F$ is of the form
$$
\left(
  \begin{matrix}
    \mathds{1}_n & & 0\\
      & 1 & \\
    0 & & t\mathds{1}_n\\
  \end{matrix}
\right)\, .
$$
Then the composition $E\otimes E\too F\otimes F \too M$
is
$$
\left(
  \begin{matrix}
    0 & & t\mathds{1}_n\\
      & t & \\
    t\mathds{1}_n & & 0\\
  \end{matrix}
\right)\, .
$$
This homomorphism $E\otimes E \too M$
 vanishes at $x\in M$, hence it factors through
$L:=M(-x)$, inducing a homomorphism $\varphi:E\otimes E \too L$.
This $\varphi$ is symmetric and nondegenerate.
\end{proof}

\begin{proposition}
\label{prop:odd2}
  Let $(E,\varphi:E\otimes E\too L)$ be an orthogonal bundle
over $X$ with $\rk(E)=2n+1$, and let $W_x\subset E_x^\vee$ be
an isotropic subspace of dimension $n$. Define $F^\vee$
using the short exact sequence
\begin{equation}
  \label{eq:odd2}
0\,\too\, F^\vee\,\too\, E^\vee\,\too\, E^\vee_x/W_x
\,\too\, 0 \, .
\end{equation}
So $F\, \subset\, E(x)\, :=\, E\otimes_{{\mathcal O}_X}{\mathcal 
O}_X(x)$. Consider the composition
$$
F\otimes F\, \hookrightarrow\, E(x)\otimes E(x)
\, \stackrel{\varphi}{\longrightarrow}\, L(2x)\, .
$$
Its image lies in $M\, :=\, L(x)\, \subset\, L(2x)$, 
and the corresponding symmetric bilinear form
$$
\psi\,:\, F \otimes F \,\too\, M
$$ 
is everywhere non-degenerate except at the point $x$, where it
is locally of the form
$$
\left(
  \begin{matrix}
    0 & & \mathds{1}_n\\
      & t & \\
    \mathds{1}_n & & 0\\
  \end{matrix}
\right)\, .
$$
\end{proposition}

\begin{proof}
With respect to a local trivialization of $E$ compatible with the 
filtration $W_x\subset S_x\subset E_x^\vee$, we have
$$
\varphi^{-1}=
\left(
  \begin{matrix}
    0 & & \mathds{1}_n\\
      & 1 & \\
    \mathds{1}_n & & 0\\
  \end{matrix}
\right)
\quad \text{~and~}\,~\,
f=
\left(
  \begin{matrix}
    \mathds{1}_n & & 0\\
      & 1 & \\
    0 & & t\mathds{1}_n\\
  \end{matrix}
\right)
$$
where $f$ is the homomorphism $F^\vee\hookrightarrow E^\vee$.
Therefore,
$$
\psi'=
\left(
  \begin{matrix}
    0 & & t\mathds{1}_n\\
      & 1 & \\
    t\mathds{1}_n & & 0\\
  \end{matrix}
\right)
\quad \text{~and~}\,~
(\psi')^{-1}=
\frac{1}{t}
\left(
  \begin{matrix}
    0 & & \mathds{1}_n\\
      & t & \\
    \mathds{1}_n & & 0\\
  \end{matrix}
\right)
$$
Since $(\psi')^{-1}$ has a pole of order one at $x\in X$,
it induces a homomorphism with values on $M:=L(x)$; this
induced homomorphism has the required properties.
\end{proof}

We will also need to consider vector bundles $F$ with a
symmetric bilinear
tensor $\psi:F\otimes F\too M$ which can be degenerate at some point
(as in Proposition \ref{prop:odd1}). In this case we still
say that a subsheaf $F'\subset F$ is \textit{isotropic} if the 
restriction
of $\psi$ to $F'\otimes F'$ is identically zero.

Following \cite{NR2}, we define:

\begin{definition}
Let $k$, $l$ be integers. A symmetric bilinear tensor $(E,\varphi)$
is $(k,l)$--stable (respectively, $(k,l)$--semistable) if for all
isotropic subbundles $E'$ of it of positive rank, the following 
inequality holds:
$$
\frac{\deg(E') + k}{\rk(E')} < \frac{\deg(E)+k-l}{\rk(E)}
$$
(respectively, $\frac{\deg(E') + k}{\rk(E')} 
\leq \frac{\deg(E)+k-l}{\rk(E)}$).
\end{definition}

If $k=l=0$ and $(E,\varphi)$ is an orthogonal bundle
(meaning $\varphi$ is nondegenerate), then the above
definition coincides with the definition of (semi)stable 
orthogonal bundles.

For any $t\,\in\, \mathbb R$, let $[t]$ be the unique integer
such that $t\leq [t] <t+1$.

\begin{lemma}
  \label{stable}
Let $(E,\varphi)$ be a $(n,n)$--stable
orthogonal bundle of rank $r$, where $n=[r/2]$.
Let $(F,\psi)$ be obtained from $(E,\varphi)$ as in
Proposition \ref{prop:even2} or Proposition \ref{prop:odd2}
(depending on the parity of $r$),
and let $(E',\varphi')$ be obtained from 
$(F,\psi)$ as in Proposition \ref{prop:even1} or 
Proposition \ref{prop:odd1}.
Then $(E',\varphi')$ is a stable orthogonal bundle.
\end{lemma}

\begin{proof}
By construction, we have a diagram
$$
\xymatrix{
 & & {0}\\
 & & {Q_x'} \ar[u]\\
{0} \ar[r] & {E} \ar[r] & {F} \ar[u]\ar[r] & {Q_x}\ar[r]& {0} \\
& & {E'} \ar[u]\\
 & & {0}\ar[u]\\
}
$$
Let $H\subset E'$ be an isotropic subbundle.
Then $H\bigcap E$ is an isotropic subsheaf of $E$,
and $\deg H - n \leq \deg H\bigcap E$, because the length
of $Q_x$ is $n$. Therefore,
$$
\frac{\deg H}{\rk H} \leq
\frac{\deg H\cap E + n}{\rk H} \leq
\frac{\deg E}{\rk E} =
\frac{\deg E'}{\rk E'}\, ,
$$
where the second inequality follows from the
$(n,n)$--stability condition on $E$.
\end{proof}

\begin{proposition}
\label{nnstable}
Let $\SM_L$ be the moduli space of stable orthogonal bundles
of rank $r$ and degree $d$ (the
line bundle $L$ is fixed). Assume that $g(X)\,>\,n$ if
$r=2n+1$ is odd, and assume that $g(X)\,>\,n+1$ if
$r=2n$ is even. Then the subset of $\SM_L$ corresponding to
$(n,n)$--stable bundles is nonempty Zariski open.   
\end{proposition}

\begin{proof}
There is a finite set of pairs $(r',d')\in \NN\times \ZZ$
such that there is a stable orthogonal bundle $(E,\varphi)\in \SM_L$
which has a quotient of rank $r'$ and degree $d'$ contradicting
the $(n,n)$--stability condition. This and the properness of
the Quot scheme together imply that the condition of being $(n,n)$--stable 
is Zariski open.

The dimension of $\SM_L$ is calculated using 
\cite[Theorem 5.9]{Ra3} and subtracting
$g(X)$, because the line bundle $L$ where the orthogonal
form takes values is fixed. More precisely,
\begin{eqnarray*}
\dim \SM_{L}& =&
(g(X)-1)\dim \GO(r,\CC) + \dim Z(\GO(r,\CC)) - g(X)\\
&=&
(g(X)-1)\frac{r^2-r}{2}\, .
\end{eqnarray*}

We will now estimate the dimension of the subset of
the moduli space $\SM_L$ corresponding to orthogonal bundles
which are not $(n,n)$--stable.

Let $(E,\varphi)$ be such
an orthogonal bundle, and let $P$ be the corresponding 
principal $\GO(r,\CC)$--bundle. An isotropic subbundle $H\subset E$
gives a reduction of structure group $P^Q\subset P$ 
to a maximal parabolic subgroup $Q\subset \GO(r,\CC)$.

Given a principal $Q$--bundle $E_Q\,\longrightarrow\, X$,
we get a principal $\GO(r,\CC)$--bundle
$E_Q\times^Q\GO(r,\CC)$ by extending the structure group.
Since the stability condition is open, a deformation of $P^Q$
as a principal $Q$--bundle will give a deformation of $P$ which
is stable but not $(n,n)$--stable. Furthermore, any deformation
of $P$ which is not $(n,n)$--stable must be of this form for some 
parabolic subgroup $Q\subset \GO(r,\CC)$. 

The tangent space of these deformations has dimension
$h^1(P^Q(\fq))-g(X)$, where $\fq$ is the Lie algebra of $Q$,
and $P^Q(\fq)$ is the adjoint vector bundle of $P^Q$. We subtract $g(X)$
because the line bundle $L$, where the orthogonal form takes values,
is fixed.

We claim that $h^0(P^Q(\fq))=1$. Indeed, on one hand we have
\begin{equation}\label{eoh}
h^0(P^Q(\fq))\geq \dim \fz(\fq)=1\, ,
\end{equation}
where ${\mathfrak z}(\fq)$ is the center of the Lie algebra $\fq$. 
On the other hand, since $P$ is a stable principal
$\GO(r,\CC)$--bundle, we have
$$
H^0(X,P(\fgo(r,\CC)))=\fz(\fgo(r,\CC)) \; ,
$$
where $P(\fgo(r,\CC))$ is the adjoint bundle of $P$, and
$\fz(\fgo(r,\CC))\, \subset\, \fgo(r,\CC)$ is the center
\cite[page 136, Proposition 3.2]{Ra1}.
Since $\fq$ is a submodule of the $Q$--module $\fgo(r,\CC)$,
the vector bundle $P^Q(\fq)$ is a subbundle of the adjoint
vector bundle $P(\fgo(r,\CC))$. Therefore,
$$
h^0(P^Q(\fq))\leq h^0(P(\fgo(r,\CC)))
= \dim \fz(\fgo(r,\CC)) = 1 \; .
$$
Combining this with \eqref{eoh} it follows that
$h^0(P^Q(\fq))\,=\,1$.

Using Riemann--Roch,
$$
h^1(P^Q(\fq))-g(X)\,=\, -\deg (P^Q(\fq)) +
(\rk(P^Q(\fq)) -1)\cdot(g(X)-1)\, .
$$
Therefore, by Lemma \ref{degrk}, 
the codimension $\codim Z$ of the subscheme
$Z\, \subset\, \SM_{L}$ defined by all
orthogonal bundles which are not $(n,n)$--stable
satisfies the inequality
\begin{eqnarray*}
  \codim Z & \geq & \dim \SM_L -(h^1(P^Q(\fq))-g)\\
& = & \dim \textup{O}(r) (g-1) + \deg P^Q(\fq) - (\rk P^Q(\fq)-1)(g-1) \\
& = &  \deg P^Q(\fq) + (\dim \textup{O}(r) - \dim\fq +1)(g-1) \\
& = & (\frac{e}{s}-\frac{d}{r})s(r-s-1) 
- \frac{3s^2-2rs+s}{2}(g-1) \\
& \geq & 
-n(r-s-1) 
- \frac{3s^2-2rs+s}{2}(g-1)\, .
\end{eqnarray*}
In the last line we have used the fact that 
$H\subset E$ contradicts the $(n,n)$--stability condition,
which translates into the inequality $e/s-d/r\geq -n/s$.

We have to show that the expression in the last line is positive.
We first assume that $r$ is odd, so we
substitute $r=2n+1$. The first summand in the last line is
$$
f_1(s):= -n(2n+1-s-1) = -2n^2+ns \geq -2n^2+n
$$
since $s\geq 1$.
The second summand becomes
$$
f_2(s):= -\big( \frac{3s^2}{2}-2ns-\frac{s}{2}\big)(g-1) 
= -\frac{3}{2}\big( s-(\frac{4n+1}{3})\big) s (g-1)\, .
$$
The graph of the function $f_2(s)$ is a parabola,
which is zero for $s=0$ and $s=(4n+1)/3$, and has
a maximum for $s=(4n+1)/6$. Therefore, the minimum
value in the interval $1\leq s\leq n$ is attained
at $s=1$. Consequently,
$$
f_2(s) \,\geq\, f_2(1)\,
= \,-\frac{3}{2}\big( 1-(\frac{4n+1}{3})\big) (g-1)\, .
$$
Finally,
$$
\codim Z \geq f_1(s)+f_2(s) \geq -2n^2+n -(1-2n)(g-1)\, ,
$$
and this is positive when $g>n$.

We now assume that $r$ is even, so we set $r=2n$.
The first summand is then
$$
f_1(s):= -n(2n-s-1) \geq -n(2n-2)
$$
since $s\geq 1$, and the second summand becomes
$$
f_2(s):= -\big( \frac{3s^2}{2}-2ns+\frac{s}{2}\big)(g-1) 
= -\frac{3}{2}\big( s-(\frac{4n-1}{3})\big) s (g-1)\, .
$$
The graph of the function $f_2(s)$ is a parabola,
which is zero for $s=0$ and $s=(4n-1)/3$, and it
is positive in between these values. 
Note that $(4n-1)/3>n$, because $n>1$.
(Recall that
we are assuming $r> 2$. Since we are in the even
case $r=2n$, this means that $n>1$.)
Hence the minimum
value in the interval $1\leq s\leq n$ is attained
at $s=1$. Therefore,
$$
f_2(s) \geq f_2(1)= -\frac{3}{2}\big( 1-(\frac{4n-1}{3})\big) (g-1)=
(2n-2)(g-1)\, .
$$
Finally,
$$
\codim Z \geq f_1(s)+f_2(s) \geq (2n-2)(g-1-n)\, ,
$$
and it is positive when $g>n+1$.
\end{proof}

\begin{lemma}
  \label{degrk}
Let $(E,\varphi:E\otimes E\too L)$ be an orthogonal bundle
with $\rk E=r$ and $\deg E=d$. Let $P$ be the corresponding
principal $\GO(r,\CC)$--bundle. Let $H\subset E$ be an isotropic
subbundle with $\deg H=e$ and $\rk H= s$.
Let $P^Q\subset P$ be the corresponding reduction of structure
group to a maximal parabolic subgroup $Q\subset \GO(r,\CC)$.
Then
\begin{eqnarray*}
  \deg P^Q(\fq)  & = & (\frac{e}{s}-\frac{d}{r})s(r-s-1)\, , \\
  \dim \fq & = &  \frac{r^2-r}{2} - \frac{2rs-3s^2-s}{2} + 1\, .  \\
\end{eqnarray*}
\end{lemma}

\begin{proof}
Let $P^Q(\fgl(r,\CC))$ be the Lie algebra bundle associated to $P^Q$ 
and the adjoint action of $Q$ on $\fgl(r,\CC)$; so,
$P^Q(\fgl(r,\CC))\cong E^\vee\otimes E$. Since $\fq$ is a
$\GO(r,\CC)$--submodule
of $\fgl(r,\CC)$, the vector bundle $P^Q(\fq)$ is a subbundle
of $E^\vee\otimes E$. The subbundle $P^Q(\fq)$ preserves the
filtration
\begin{equation}
  \label{eq:filt}
  H\subset H^\perp \subset E \; ,
\end{equation}
where $H^\perp$ is the orthogonal bundle to $H$ with respect to
the orthogonal structure $\varphi$. Therefore, we have $\rk H + \rk 
H^\perp\,=\, r$.

Let $L(Q)$ be the Levi quotient of the parabolic subgroup $Q$.
Fixing $T\, \subset\, B\,\subset\, Q$, where $T$ is a maximal torus
of $\GO(r,\CC)$ and $B$ a Borel subgroup of $\GO(r,\CC)$, 
the quotient $L(Q)$
of $Q$ can be realized as a subgroup of $Q$. In fact, the maximal
connected $T$--invariant reductive subgroup of $Q$ is identified
with $L(Q)$. Fix such a subgroup of $Q$. This subgroup of $Q$ will
also be denoted by $L(Q)$; it will be called the \textit{Levi
subgroup}.

Let $P^{L(Q)}$ denote the principal $L(Q)$--bundle
obtained by extending the structure group of $P^Q$ using the projection
of $Q$ to its Levi quotient $L(Q)$. Let $P^{L(Q)}(Q)$ be the principal
$Q$--bundle obtained by extending the structure group of $P^{L(Q)}$
using the inclusion of the Levi
subgroup $L(Q)\, \subset\, Q$ that has been fixed. The 
principal $Q$--bundle $P^{L(Q)}(Q)$ is topologically
isomorphic to the principal $Q$--bundle $P^Q$.
Hence the two adjoint bundles $P^Q(\fq)$ and $P^{L(Q)}(Q)(\fq)$ are
topologically isomorphic. Therefore, to calculate the degree of
$P^Q(\fq)$, we can replace $P^Q$ by $P^{L(Q)}(Q)$. In other words, we
can assume that $P^Q$ admits a reduction of structure group
$P^{L(Q)}\, \subset\, P^Q$ to the Levi subgroup $L(Q)\, \subset\, Q$.
Fix a reduction of structure group
$P^{L(Q)}\, \subset\, P^Q$ to $L(Q)$.

The filtration \eqref{eq:filt} splits using the reduction of structure
group $P^{L(Q)}\, \subset\, P^Q$. In other words, we have an
isomorphism,
\begin{equation}\label{isom.-2}
E\,\cong\, H \oplus (H^\perp/H) \oplus
(E/H^\perp)\, .
\end{equation}
Using \eqref{isom.-2}, a locally defined section of the adjoint bundle
$P^Q(\fq)$ has the form
\begin{equation}\label{eq.-A}
A=\left(
\begin{array}{ccc}
\alpha & \beta& \gamma\\
0 & \delta& \epsilon\\
0 & 0 &\eta
\end{array}
\right)
\end{equation}

The isomorphism of vector bundles
\begin{equation}\label{iso.-pr.}
 E\,\too\, E^\vee\otimes L
\end{equation}
induced by the orthogonal structure
$\varphi$ has the property that the composition
$$
H^\perp\, \hookrightarrow\, F\, {\too}\,
F^\vee\otimes M\, \too\, H^\vee\otimes M
$$
vanishes. Consequently, we have an induced isomorphism
$$
F/H^\perp\, \cong\, H^\vee\otimes M\, ,
$$
which we will denote by $\mathds{1}$. Also note that $\varphi$ induces
an orthogonal structure on the vector bundle $H^\perp/H$. Let
$$
\varphi'\,:\, H^\perp/H\,\too\, (H^\perp/H)^\vee\otimes M
$$
be the isomorphism induced by this
orthogonal structure.

Now, using \eqref{isom.-2}, the isomorphism
in \eqref{iso.-pr.} has the form
\begin{equation}\label{eq.-va.-ph}
\left(
\begin{array}{ccc}
0&0&\mathds{1}\\
0& \varphi' & 0\\
\mathds{1}&0&0
\end{array}
\right)
\end{equation}
where $\varphi'$ is defined above.

A parabolic subalgebra $\fq$ of $\fgo(r,\CC)$ is of the form
$\fq=\fq'\oplus \CC$, where $\fq'$ is a
parabolic subalgebra of $\fso(r,\CC)\,=\, {\rm Lie}(\text{SO}(r,\CC))$,
and
the summand $\CC$ is the center of $\fgo(r,\CC)$. This decomposition
is preserved by the adjoint action of $\GO(r,\CC)$.
Therefore,
$$
P^Q(\fq)\,=\, P^Q(\fq')\oplus \SO_X\, .
$$
The condition that the local section $A$,
defined in \eqref{eq.-A}, of $P(\fgl(2n,\CC))$ lies in
$P^Q(\fq')$ is equivalent to the condition that
$$
\varphi\circ A\,=\,
\left(
\begin{array}{ccc}
0& 0& \eta\\
0& \varphi\circ \delta& \varphi\circ \epsilon\\
\alpha& \beta& \gamma
\end{array} \right) \,:\,
E \,\too\, E^\vee\otimes L
$$
is skew-symmetric, where $\varphi$ is defined in \eqref{eq.-va.-ph}.
Clearly, $\varphi\circ A$ is skew-symmetric if and only if the 
following three conditions hold:
\begin{enumerate}
\item $\eta\,=\,-\alpha^t$,

\item $\epsilon\,=-\,\varphi'{}^{-1}\circ \beta^t$, and

\item the homomorphisms $\gamma$ and $\varphi'\circ \delta$ are
skew-symmetric.
\end{enumerate}
Therefore, there is an isomorphism
\begin{equation}\label{cal.-deg.}
P^Q(\fq') \,\cong\, \text{End}(H)
\oplus ((\frac{H^\perp}{H})^\vee\otimes H) \oplus
((\bigwedge^2 H)\otimes M^\vee) \oplus (\bigwedge^2
(\frac{H^\perp}{H})^\vee \otimes M)
\end{equation}
defined by
$$
A \, \longmapsto\,
(\alpha\, , \beta\, ,\gamma\, ,\varphi'\circ\delta)\, .
$$
{}From this isomorphism it follows immediately that
$\rk (P^Q(\fq))\,=\,\dim \fq$.

Using \eqref{cal.-deg.} we further have
$$
\deg (P^Q(\fq))\,=\, \deg(P^Q(\fq'))\,=\,
(\frac{e}{s}-\frac{d}{r})s(r-s-1)\, .
$$
This completes the proof.
\end{proof}

\section{Hecke cycles}

In this section, we will construct a family
of orthogonal bundles using the constructions introduced
in Proposition \ref{prop:even1} and Proposition \ref{prop:odd1}.
If the rank is even, then the starting point will be
an orthogonal bundle. If the rank is odd, then
the starting point will be a vector bundle with
a symmetric bilinear form singular over a fixed point $x$
(as in Proposition \ref{prop:odd1}).

We will first describe the even case $r=2n$.

Let $(F,\psi:F\otimes F\too L(x))$ be a $(0,n)$--stable
orthogonal bundle over $X$ of rank $r=2n$
(as in Proposition \ref{prop:even1}). So
$\psi$ takes values in
the line bundle $L(x)$, where $x\in X$ is the fixed point.
We assume that $L$ is such that
\begin{equation}\label{dF}
\deg F>(2g-2)r+r+n\, .
\end{equation}

Take a quotient $F_x\,\too\, Q_x$ over the given point $x\in X$
with $\dim Q_x=n$. We obtain the
following commutative diagram of coherent sheaves on $X$:
\begin{equation}
\label{com}
\xymatrix{
& {0} & {0} \\
{0} \ar[r] & {S_x} \ar[r]\ar[u] & {F_x}\ar[r]\ar[u] & {Q_x}\ar[r]& {0}\\
{0} \ar[r] & {E} \ar[r]\ar[u]  & {F}\ar[r]\ar[u] & {Q_x}\ar[r]\ar@{=}[u]& {0}\\
& {F(-x)} \ar[u] \ar@{=}[r]& {F(-x)}\ar[u] \\
& {0}\ar[u] & {0}\ar[u] \\
}\end{equation}
All the quotients of $F_x$ with the property that the
corresponding kernel $S_x$ is an isotropic
subspace of $F_x$ are parameterized by
$$
Y\,=\,\text{Gr}_{\text{iso},n}(F_x)\,\cong\,\text{GO}(2n,\CC)/Q\, ,
$$
where $Q\, \subset\, \text{Gp}(2n,\CC)$ 
is the parabolic subgroup preserving a fixed
isotropic subspace of dimension $n$ of $\CC^{2n}$. 

All these Hecke
transformations parameterized by $Y$ combine to form
the following commutative diagram of sheaves on $X\times Y$
\begin{equation}
\label{family}
\xymatrix{
& {0} & {0} \\
{0} \ar[r] & {i_*j^*\cS} \ar[r]\ar[u] & 
{F_x\otimes \SO_{\{x\}\times Y}}\ar[r]\ar[u] & 
{i_*j^*\SQ}\ar[r]& {0}\\
{0} \ar[r] & {\SE} \ar[r]\ar[u]  & {p_1^*F}\ar[r]\ar[u] & 
{i_*j^*\SQ}\ar[r]\ar@{=}[u]& {0}\\
& {p^*_1 F(-x)} \ar[u] \ar@{=}[r]& {p^*_1 F(-x)}\ar[u] \\
& {0}\ar[u] & {0}\ar[u] \\
}
\end{equation}
where $i\,:\,\{x\}\times Y\,\too\
X\times Y$ and $j\,:\,\{x\}\times Y \,\too\, Y$
are the natural inclusion and isomorphism respectively;
here $p_1$ is the natural projection of $X\times Y$ to $X$.
The vector bundles $\cS$ and $\SQ$ are respectively the 
universal subbundle
and quotient bundle of $F_x$ over $Y\cong \GO(2n,\CC)/Q$.

If the given orthogonal bundle
$(F,\psi:F\otimes F\too L(x))$ is $(0,n)$--stable,
then all the orthogonal bundles
constructed by Hecke transformations from
$(F,\psi)$ are stable (see Lemma \ref{stable}).
The resulting classifying morphism
\begin{equation}
\label{heckemorphism}
\Psi\,:\,Y \,\too\, \SM_L
\end{equation}
will be called the \textit{Hecke morphism}, where
$\SM_L$ as before is the moduli space of orthogonal bundles
of rank $2n$ with values in $L$.

We claim that $H^1(X, F(-x))\,=\,0$.
Indeed, after twisting the middle row of (\ref{com})
with $\SO_X(-x)$, the associated long exact sequence
gives a surjection $$H^1(E(-x))\too H^1(F(-x))\, .$$
Since $(F,\psi:F\otimes F\too L(x))$ is $(0,n)$--stable,
the orthogonal bundle $(E,\varphi:E\otimes E\too L)$ is stable,
therefore the underlying vector bundle 
$E$ is semistable by Proposition \ref{Esemistable}.
We have $$h^1(E(-x))\,=\, h^0(K_X(x)\otimes E^\vee)$$
(Serre duality). On the other hand, 
$\deg E=\deg F-n > (2g-2)r+r$ (see \eqref{dF}).
Therefore, $K_X(x)\otimes E^\vee$ is
a semistable vector bundle of negative degree,
so it cannot have nonzero sections. This
proves the claim that $H^1(X, F(-x))\,=\,0$.

Applying $p_{2*}$ to \eqref{family}, where
$p_2$ is the natural projection of $X\times Y$ to $Y$,
we obtain the following
commutative diagram of sheaves on $Y$:
\begin{equation}
\label{globalcom}
\xymatrix{
& {0} & {0} \\
{0} \ar[r] & {\cS} \ar[r]\ar[u] & 
{F_x\otimes \SO_{Y}}\ar[r]\ar[u] & 
{\SQ}\ar[r]& {0}\\
{0} \ar[r] & {\SW=p_{2*}\SE} 
\ar[r]\ar[u]  & {H^0(F)\otimes \SO_Y}\ar[r]\ar[u] & 
{\SQ}\ar[r]\ar@{=}[u]& {0}\\
& {H^0(F(-x))\otimes \SO_Y} \ar[u] \ar@{=}[r]& 
{H^0(F(-x))\otimes \SO_Y}\ar[u] \\
& {0}\ar[u] & {0}\ar[u] \\
}
\end{equation}

Now assume that $r=2n+1$.

Let $F\,\longrightarrow X$ be a vector bundle of
rank $r=2n+1$ with a symmetric
bilinear form
$$
\psi:F\otimes F\too L(x)
$$
which induces a short exact sequence as in \eqref{exs}
(so $(F\, ,\psi)$ is as in Proposition \ref{prop:odd1}).
We assume that $(F\, ,\psi)$ is $(0,n)$--stable, and also
assume that \eqref{dF} holds.

Let $S_x\subset F_x/l_x$
be an isotropic subspace of dimension $n$.
The isotropic subspace of dimension $n$ are parameterized by
$$
Y   \,=\, \text{Gr}_{\text{iso},n}(F_x/l_x)
   \,\cong\, \text{GO}(2n,\CC)/Q .
$$
This subspace $S_x$ induces a commutative diagram
\begin{equation}
\label{com'}
\xymatrix{
& {0} & {0} \\
{0} \ar[r] & {S_x} \ar[r]\ar[u] & {F_x/l_x}\ar[r]\ar[u] & {Q_x}\ar[r]& {0}\\
{0} \ar[r] & {E} \ar[r]\ar[u]  & {F}\ar[r]\ar[u] & {Q_x}\ar[r]\ar@{=}[u]& {0}\\
& {F'} \ar[u] \ar@{=}[r]& {F'}\ar[u] \\
& {0}\ar[u] & {0}\ar[u] \\
}\end{equation}
Note that this diagram is different from \eqref{com}, because
in the top row we have $F_x/l_x$ instead of $F_x$;
in the bottom row, instead of $F(-x)$, we have a new vector
bundle $F'$.
Arguing as in the even case, we finally obtain a commutative
diagram on $Y$, with $\cS$ and $\SQ$ respectively being the
universal isotropic subbundle and quotient of $F_x/l_x$
\begin{equation}
\label{globalcom'}
\xymatrix{
& {0} & {0} \\
{0} \ar[r] & {\cS} \ar[r]\ar[u] & 
{F_x/l_x\otimes \SO_{Y}}\ar[r]\ar[u] & 
{\SQ}\ar[r]& {0}\\
{0} \ar[r] & {\SW=p_{2*}\SE} 
\ar[r]\ar[u]  & {H^0(F)\otimes \SO_Y}\ar[r]\ar[u] & 
{\SQ}\ar[r]\ar@{=}[u]& {0}\\
& {H^0(F')\otimes \SO_Y} \ar[u] \ar@{=}[r]& 
{H^0(F')\otimes \SO_Y}\ar[u] \\
& {0}\ar[u] & {0}\ar[u] \\
}
\end{equation}

\begin{proposition}
\label{stables}
Let $Q\subset \GO(2n,\CC)$ be the maximal parabolic subgroup
preserving a fixed isotropic subbundle $V\subset \CC^{2n}$
of dimension $n$.
Then the Picard group of $Y\,=\,\GO(2n,\CC)/Q$ is $\ZZ$,
the universal vector subbundle $\cS\, \longrightarrow\,
\GO(2n,\CC)/Q$ is stable,
and the determinant of $\cS$ is $-2 A$, where $A$ is the ample
generator of $\Pic(Y)$.
\end{proposition}

\begin{proof}
Note that
$$
\GO(2n,\CC)/Q=\text{SO}(2n,\CC)/Q'=\text{Spin}(2n,\CC)/\wt{Q}'\, ,
$$
where $Q'=Q\bigcap \text{SO}(2n,\CC)$, and $\wt{Q}'$ is
the corresponding maximal parabolic subgroup in 
$\text{Spin}(2n,\CC)$. Therefore, $\wt{Q}'\too Q'$
is a 2-to-1 covering which restricts to $t\longmapsto t^2$
on the center.

The spin group is semisimple and simply connected, hence
the Picard group of $\textup{Spin}(2n,\CC)/\wt{Q}'$ is identified
with the character group of $\wt{Q}'$, which is equal to the
character group of the Levi quotient $L(\wt{Q}')$; since
$\wt{Q}'$ is maximal, this character group is equal to $\ZZ$.
Therefore, $\Pic Y=\ZZ$ (\cite{BH1, BH2}).

Take another isotropic subspace 
$W \subset \CC^{2n}$
such that $V+ W = \CC^{2n}$. The orthogonal form on $\CC^{2n}$
identifies $W$ with $V^\vee$. A Levi subgroup $L(Q')$ of $Q'$
is defined
by all orthogonal automorphisms of $\CC^{2n}$ that preserve
the direct sum decomposition $V\oplus W$, i.e.,
automorphisms taking $V$ to $V$ and $W$ to $W$.

Since $W \,=\, V^\vee$, the Levi subgroup
$L(Q')\, \subset\, Q'$ is identified with $\text{GL}(V)$,
sitting inside $\text{SO}(2n,\CC)$ as
\begin{equation}
\label{inj}
A \,\longmapsto\, 
\left(
\begin{array}{cc}
A& 0\\
0& (A^t)^{-1}
\end{array}
\right)
\end{equation}
using the above decomposition $\CC^{2n} = V\oplus V^\vee$.
Consequently, the character $f$ of $L(Q')\, =\,
\text{GL}(V)$ defined by
$$
f(A)\,\longmapsto\, \det A\, ,
$$
where $A$ is as in \eqref{inj}, generates the character group
of $L(Q')$. 

Since $\wt{Q}'\too Q'$ is a 2-to-1
covering, restricting to $t\longmapsto t^2$ on the center,
the character $f$ corresponds to twice the 
generator of the character group of $L(\wt{Q}')$.
The line bundle over $Y$ defined by the above character
$f$ of $L(Q')$
coincides with the top exterior power of the tautological
subbundle $\cS$.
Therefore, the first Chern class of the vector
bundle $\cS$ over the complete homogeneous
space $Y$ is equal to twice a generator 
of the Picard group of $Y$. It is easy to see that
$c_1(\cS)$ is non-positive.

Since $\cS$ is identified with the vector bundle associated 
to the principal $Q'$--bundle over $Y$ for
an irreducible representation of the Levi quotient of $Q'$,
a theorem due to Ramanan and Umemura says that the vector
bundle $\cS$ is stable (see \cite[page 136, Theorem 2.4]{U}).
\end{proof}

\begin{corollary}
\label{nonzero}
The Hecke morphism $\Psi:Y\too \SM_L$ in
\eqref{heckemorphism} induces a nonzero 
homomorphism $\Psi^*:\Pic \SM_L \too \Pic Y$.
\end{corollary}

\begin{proof}
Since $(F,\psi)$ is $(0,n)$--stable, we obtain a family of 
stable orthogonal bundles $(E,\varphi)$, and the
morphism $\Psi$ is well defined.

In the even case (respectively, odd case),
it follows from \eqref{globalcom} (respectively, \eqref{globalcom'})
that $\deg \SW\,=\,\deg \cS$, which by Proposition
\ref{stables} is equal to $-2$. But if the induced homomorphism
$\Psi^*$ were zero, then we would have had $\deg \SW=0$. 
\end{proof}

\section{Stability of Picard bundle}
\label{picard}

Let
\begin{equation}\label{u-a}
\SM^0_L\, \subset\, \SM_L
\end{equation}
be the locus of stable principal bundles
for which the automorphism group coincides with the
center of $\GO(r,\CC)$.

\begin{lemma}\label{lem-aut}
The subset $\SM^0_L$ in \eqref{u-a} is Zariski open, and
its complement is of codimension at least two.
\end{lemma}

\begin{proof}
Let $Z(\GO(r,\CC))\, \subset\, \GO(r,\CC)$ be the center.
Let $E_{\GO(r,\CC)}\, \in\, \SM_L$ be a stable
principal $\GO(r,\CC)$--bundle. Assume that
the automorphism group $\text{Aut}(E_{\GO(r,\CC)})$
has an element $\tau$ which does not lie in
$Z(\GO(r,\CC))$. The Lie algebra of 
$\text{Aut}(E_{\GO(r,\CC)})$ coincides
with the global section
$H^0(X, E_{\GO(r,\CC)}(\fgo(r,\CC)))$ of the
adjoint vector bundle, and hence it coincides
with the center $\fz(\fgo(r,\CC))$ because
$E_{\GO(r,\CC)}$ is stable. Therefore,
the quotient $\text{Aut}(E_{\GO(r,\CC)})
/Z(\GO(r,\CC))$ is a finite group.
This implies that $\tau$ is a semisimple element.

Since $\tau$ is semisimple, it defines a conjugacy
class of elements of $\GO(r,\CC)$ (see the second
paragraph of \cite[Section~3]{BBN}). Fix an
element $\overline{\tau}\, \in\, \GO(r,\CC)$
in the conjugacy class given by $\tau$. Let
$$
{\mathcal C}_{\overline{\tau}}\, \subset\, \GO(r,\CC)
$$
be the centralizer of
$\overline{\tau}$. The principal $\GO(r,\CC)$--bundle 
$E_{\GO(r,\CC)}$ admits a holomorphic reduction of structure
group to ${\mathcal C}_{\overline{\tau}}$ \cite[p. 230, Theorem 
3.2]{BBN}. Since the principal $\GO(r,\CC)$--bundle
$E_{\GO(r,\CC)}$ is stable, it does not admit any reduction
of structure group to any Levi subgroup of some proper
parabolic subgroup of
$\GO(r,\CC)$. Hence ${\mathcal C}_{\overline{\tau}}$ is not
a Levi subgroup of some proper parabolic subgroup of $\GO(r,\CC)$.

Up to conjugacy, there are only finitely many semisimple
elements
$$
c'_1\, ,\cdots \, ,c'_m\, \in\,
\GO(r,\CC)/Z(\GO(r,\CC))
$$
whose centralizer is not a Levi subgroup of some parabolic
subgroup of the semisimple group $\GO(r,\CC)/Z(\GO(r,\CC))$ (see 
\cite[p.~113]{DM}).
Fix elements
$$
c_1\, ,\cdots \, ,c_m\, \in\, \GO(r,\CC)
$$
such that $c_i$ projects to $c'_i$.

Let
$$
{\mathcal C}_{c_i}\, \subset\, \GO(r,\CC)
$$
be the centralizer of $c_i$. Let
${\mathcal M}({\mathcal C}_{c_i})$ be the moduli space of
stable principal ${\mathcal C}_{c_i}$--bundles over $X$
that maps to $\SM_L$ by extension of structure group
of principal ${\mathcal C}_{c_i}$--bundles to
$\GO(r,\CC)$. Note that the complement 
$\SM_L\setminus \SM^0_L$ is the image of the union
$\bigcup_i {\mathcal M}({\mathcal C}_{c_i})$. From the
formula for dimension of a moduli
space of principal bundles it follows immediately that
$$
\dim {\mathcal M}({\mathcal C}_{c_i})\, \leq\,
\dim \SM_L -2\, .
$$
This completes the proof of the lemma.
\end{proof}

The projectivized Picard bundle is a
principal $\PGL(N,\CC)$--bundle $P^\PGL$ on $\SM^0_L$,
such that for any point $(E,\varphi)\, \in\, \SM^0_L$,
the fiber over $(E,\varphi)$ of the associated projective
bundle $P^\PGL(\PP^{N-1})$ is canonically
identified with $\PP(H^0(X, E))$. Note that $N\, =\,
\dim H^0(X, E)$. From the construction of the moduli space
it follows that the projectivized Picard bundle exists.

The moduli space $\SM_L$ has a natural compactification,
namely the coarse moduli space of semistable orthogonal
bundles, which known to be a normal projective variety. The
complement $\overline{\SM}_L\subset \SM_L$ has codimension
at least 2, and hence we can think of the projectivized
Picard bundle as a rational principal bundle (their definition
is recalled below) on the projective variety $\overline{\SM}_L$.

Recall from \cite{Ra2} 
that a \textit{rational principal bundle} on a normal projective
variety $M$ is a principal bundle $P$ on a big open set
$U\subset M$
(i.e., an open set whose complement has codimension at least two).
A rational principal bundle
is said to be \textit{stable} (respectively, \textit{semistable}) 
with respect to a polarization $\SO_M(1)$
if for every reduction $P^Q\subset P|_V$ to a maximal 
parabolic subgroup $Q$ defined on big open subset
$V\subset U$, and for all nontrivial dominant characters of $Q$
which are trivial on the center of $G$,
the inequality
$\deg P^Q(\chi)<0$ (respectively, $\deg P^Q(\chi)\leq 0$) holds, 
where $P^Q(\chi)=(P^Q\times \CC_\chi)/Q$
is the line bundle over $V$ associated to $P^Q$ for
the character $\chi$, and the degree is
calculated with respect to the polarization $\SO_M(1)$.

\begin{proposition}
\label{general}
Assume that
$g(X)\,>\,n$ if $r=2n+1$ is odd, and $g(X)\,>\,n+1$ if
$r=2n$ is even. 
Fix distinct points $x_1,\cdots,x_m\in X$. Let $Z$ be a subscheme of 
$\SM_L$ of codimension at least two. Then there is a nonempty Zariski 
open subset $V_0\,\subset\, \SM^0_L\,\subset \, \SM_L$ such that the 
following two hold:
\begin{enumerate}
\item If $(E,\varphi)\,\in\, V_0$, is $(n,n)$--stable.

\item Take any $(E,\varphi)\,\in\, V_0$ and choose
a point $x_i$. Let
$(F,\psi)$ be a general Hecke transformation of $(E,\varphi)$
with respect to $x_i$.
If $\Psi\,:\, Y \,\too\,\SM_L$ 
is the Hecke morphism, then $\Psi^{-1}(Z)$ is either
empty or its codimension in $Y$ is at least two.
\end{enumerate}
\end{proposition}

\begin{proof}
We will first assume that $m=1$, so there is only one point $x_1=x$.
If $m>1$, we take the intersection of the open subsets of
$\SM^0_L$ corresponding to each point $x_i$. 

Let $\SM_L^{(n,n)}\subset \SM^0_L$ be the open subset 
of $(n,n)$--stable orthogonal bundles. This is dense
because of Proposition \ref{nnstable}.

Let $p:\SY\too \SM^{(n,n)}_L$ be the fibration whose
fiber over $(E,\varphi)$ is canonically isomorphic
to $\text{Gr}_{\text{iso},n}(E_x^\vee)$, i.e., the set
of isotropic subspaces of $E_x^\vee$ of dimension $n$.
This fibration can easily be constructed because 
all the points in $\SM^{(n,n)}_L\subset \SM^0_L$, by definition, 
correspond to orthogonal bundles whose automorphisms are scalars.
Therefore, by Proposition \ref{prop:even2} or
Proposition \ref{prop:odd2}, each point of $\SY$ corresponds to a 
short exact sequence
\begin{equation}
\label{pointy}
0 \too E \too F \too Q_x \too 0\, ,
\end{equation}
where $F$ is the corresponding Hecke transform;
indeed, if we apply $Hom(\cdot\, ,\SO_X)$ to 
(\ref{eq:even2}) or (\ref{eq:odd2}) we obtain
this exact sequence).

Let $q:\ST\too \SY$ be the fibration whose fiber
over a point corresponding to a short exact sequence
as in \eqref{pointy} is canonically isomorphic
to $\text{Gr}_{\text{iso},n}(F_x/l_x)$ if $r$ is odd,
or canonically isomorphic to $\text{Gr}_{\text{iso},n}(F_x)$
if $r$ is even. Therefore, each point of $\ST$ corresponds to
a diagram
\begin{equation}
  \label{eq:pointt}
\xymatrix{
 & & {0}\\
 & & {Q_x'} \ar[u]\\
{0} \ar[r] & {E} \ar[r] & {F} \ar[u]\ar[r] & {Q_x}\ar[r]& {0} \\
& & {E'} \ar[u]\\
 & & {0}\ar[u]\\
}
\end{equation}
For all points in $\ST$, the orthogonal bundle $(E,\varphi)$ 
is $(n,n)$--stable, hence $(F,\psi)$ is $(0,n)$--stable,
and then $(E',\varphi')$ is stable as an orthogonal
bundle. Therefore, by sending 
a point corresponding to a diagram in \eqref{eq:pointt} to
$(E',\varphi')$, we obtain a morphism $p':\ST\too\SM_L$. So,
$$
\xymatrix{
 & {\ST} \ar[dr]^{q} \ar[ddl]_{p'}\\
& & {\SY}\ar[d]^{p}\\
{\SM_L}& & {\SM_L^{(n,n)}}
}
$$
Note that, for each point $y\in \SY$, the fiber $q^{-1}(y)$
is identified with a Hecke cycle $Y$ which is mapped by
$p'$ to $\SM_L$.

Let $Z'=p'{}^{-1}(Z)$. If $q(Z')$ is not dense in $\SY$, take
an open subset in the complement; the image of this open set
under $p$ satisfies the condition of the proposition.

It remains to analyze the case when $q(Z')$ is dense
in $\SY$. In this
case, the dimension of the generic fiber of $q|_{Z'}$ is
$$
\dim Z' - \dim \SY = \dim Z + \dim p' - \dim \SY = 
\dim Z + \dim p' - \dim \SM^{(n,n)}_L - \dim p = 
$$
$$
(\dim p - \dim p') +(\dim Z - \dim \SM_L)  \leq
\dim Y - 2\, .
$$
This completes the proof of the proposition.
\end{proof}

We have $\Pic(\SM^0_L)=\Pic(\SM^{}_L)=\ZZ$ (see \cite{BHo}, \cite{BLS}), 
so $\SM^0_L$ has a unique polarization. 

We can now state and prove the main theorem. Let $P^\PGL$ be the 
above defined projectivized Picard bundle on the moduli
space $\SM^0_L$ of stable orthogonal bundles
$$
(E,\varphi:E\otimes E\too L)
$$
on $X$. Denote $r=\rk E=r$, and $d=\deg E$.

\begin{theorem}\label{thm.}
Assume that
$g(X)\,>\,n$ if $r=2n+1$ is odd, and $g(X)\,>\,n+1$ if
$r=2n$ is even. Also assume that $d>(2g-2)r+r$.
Then, the projectivized Picard bundle $P^\PGL$
over the moduli space $\SM^0_L$
is stable (since the Picard group of $\SM^0_L$ is $\mathbb Z$,
stability is independent of choice of polarization).
\end{theorem}

\begin{proof}
Let
\begin{equation}\label{de.-re.}
P^{\overline{Q}}\, \subset\,P^\PGL
\end{equation}
be a reduction of structure group of $P$, on a big Zariski
open set $U\,\subset\,\SM^0_L$, to a maximal parabolic subgroup
$\overline{Q}\, \subset\, \PGL(N,\CC)$
(recall that by a big Zariski open subset we mean one whose 
complement is of codimension
at least two). The parabolic subgroup $\overline{Q}$
is the image of a unique maximal parabolic subgroup $Q$ of
$\GL(N,\CC)$ by the
natural projection $\GL(N,\CC)\,\too\,\PGL(N,\CC)$.
We need to prove that for a nontrivial dominant character
$\chi$ of $\overline{Q}$, the inequality
\begin{equation}
\label{stability}
\deg(P^{\overline{Q}}(\chi))\,<\, 0
\end{equation}
holds.
If $A$ is the unique proper nonzero subspace of $\CC^N$
preserved by $Q$, then any nontrivial dominant character
of $\overline{Q}$ is a positive multiple of the character
defined by the natural action of $\overline{Q}$ on the
line $\bigwedge^{\text{top}} \text{Hom}(\CC^N/A\, ,A)$.

The strategy of the proof is to use a Hecke morphism $\Psi\,:
\, Y\,\too\, \SM_L$ 
(defined in \eqref{heckemorphism}). We denote the restriction to
$Y^0:=\Psi^{-1}(\SM^0_L)$ as $\Psi^0:Y^0\too \SM^0_L$. Then we
calculate the degree using the pullback of $P^{\overline{Q}}$
to an open subset of $Y^0$. For this to work, we need that the
Hecke cycle be ``general enough'' in the sense that
$\Psi^{-1}(\SM_L\setminus U)$ has codimension
at least two. This is to ensure that the inclusion map
$Y^0\hookrightarrow Y$ induces an isomorphism of Picard groups.

Fix $m$ distinct points
$x_1,\cdots,x_m\in X$ with $m\,>\,\deg(E)(1+ 1/\rk(E))$. Set
$$
Z\, =\, U^c\, =\, \SM_L\setminus U \; .
$$
(Recall that the Picard bundle is defined on $\SM^0_L \subset\SM_L$,
and the reduction $P^{\overline{Q}}$ is defined on $U\subset \SM^0_L$.)
Fix an orthogonal bundle $(E,\varphi)$ over $X$
corresponding to a point in the intersection
of $U$ with the 
open subset of $\SM^0_L$ given by Proposition \ref{general}.

The reduction of structure group $P^{\overline{Q}}$ gives a
projective subbundle 
$$
P^{\overline{Q}}(\PP^{N'-1})\,\subset\, P^\PGL(\PP^{N-1})
$$
over $U$. 
The fiber of $P^\PGL(\PP^{N-1})$ over $(E,\varphi)$ is 
canonically isomorphic to $\PP(H^0(E))$, and the fiber
of $P^{\overline{Q}}(\PP^{N'-1})$ defines a subspace 
$V'\subset H^0(E)$.

Fix a nonzero element
$$
s\,\in\, V'\subset H^0(X,E)
$$
which we are going to consider as a section of $E$.

We claim that the section $s$ of $E$ 
cannot vanish in more than $\deg(E)/\rk(E)$
points. Indeed, if $D$ is a subset
of $\{x_1,\cdots,x_m\}$, and $s$ vanishes
in all points of $D$, then the section
$s: {\mathcal O}_X\longrightarrow E$ factors through $\SO_X(D)$, 
and the semistability of $E$ (Proposition \ref{Esemistable}) 
implies that $\deg(D)\leq \deg(E)/\rk(E)$.

Analogously, if $E^{V'}$ is the subsheaf of $E$ generated
by the sections $V'\subset H^0(X,E)$, then $E^{V'}$ cannot
fail to be a subbundle in more than $\deg(E)$ points.
Indeed, if $D$ is a subset of $\{x_1,\cdots,x_m\}$, and
$E^{V'}$ is not a subbundle on all points of $D$,
then the inclusion $E^{V'}\subset E$ factors through
a subsheaf $\wt{E}\subset E$ with $$\deg(\wt{E})\geq \deg(E^{V'})
+\deg D\, .$$
Since $E^{V'}$ is generated by global sections, it follows that
$\deg(E^{V'})\geq 0$. On the other hand, the stability condition
of $E$ implies that $$\deg \wt{E}\leq \deg(E)\rk(\wt{E})/\rk(E)\leq 
\deg(E)\, ,$$
therefore $\deg(D)\leq \deg(E)$.

Consequently, we can choose a point $x\in \{x_1,\cdots,x_m\}$
such that $s(x)\neq 0$, the sheaf $E^{V'}$ generated by $V'$ 
is locally free at $x$,
and the induced homomorphism 
\begin{equation}
  \label{eq:inj}
E^{V'}_x\too E_x^{}
\end{equation}
is injective.

If $r=2n$, 
using Proposition \ref{general}, we can choose an isotropic
subspace
$S_x\subset E_x$ of dimension $n$ such that
\begin{equation}\label{de.-ell}
s(x)\,\notin\, S_x\, ,
\end{equation}
and the subscheme $\Psi^{-1}(Z)\,\subset\, Y$ is either
empty or it is of
codimension at least two 
in $Y=\text{Gr}_{\text{iso},n}(F_{x})$, where $\Psi:Y\too
\SM_L$ is the 
Hecke morphism associated to $(F,\psi)$ and the
point $x\in X$ (see \eqref{heckemorphism}).

If $r=2n+1$, using again Proposition \ref{general}, we
choose an isotropic subspace $W_x \subset E_x$ of
dimension $n$ such that $s(x)\notin W_x^\perp$.

This bound on the codimension of $\Psi^{-1}(Z)$
implies that $\Pic(Y)=\Pic(Y_U)$, where $Y_U\, :=\,
Y\setminus \Psi^{-1}(Z)$ is the complement.
Now using Corollary \ref{nonzero} 
it follows that the homomorphism
$$
\Psi^*_U\,:\, \Pic(\SM_L)\,\too\, \Pic (Y_U)\,=\,\Pic(Y)
$$
is nonzero. Therefore,
there is a positive rational number $k$ such that 
\begin{equation}
  \label{eq:k}
\deg(P^{\overline{Q}}(\chi))
\,=\, k\cdot
\deg(\Psi^*_U P^{\overline{Q}}(\chi))\, .
\end{equation}
Let $\SE$ be the family
of Hecke transformations parameterized by $Y$ (see \eqref{family}). 
The pulled back projective bundle $\Psi^*P^\PGL$
on $Y$ lifts to the principal $\GL(N,\CC)$--bundle
$P^\GL$ over $Y$ associated to
$$
\SW\,=\, p_{2*}\SE\, ,
$$
where $p_{2}:X\times Y\too Y$ is the
natural projection. This means that there is an isomorphism
\begin{equation}\label{d.-e.}
\PP(\SW)\,\cong\, P^\PGL(\PP^{N-1})
\, .
\end{equation}
 
Let $\Psi_U\,:\,Y_U\,\too\, U$ be the restriction
of $\Psi$ to $Y_U\,=\,\Psi^{-1}(U)$. We claim that the pullback
$\Psi_U^* P^{\overline{Q}}$ of the reduction
in \eqref{de.-re.} is given by a subbundle of $\SW$. In other
words, there is a subbundle
\begin{equation}
  \label{eq:f}
\SH\,\inj\,\SW 
\end{equation}
on $Y_U$
such that the subbundle $\PP(\SH)\,\subset\,
\PP(\SW)$ is identified with the subbundle
$P^{\overline{Q}}(\PP^{N'-1})
\, \subset\, P^{\overline{Q}}(\PP^{N-1})$ by the
isomorphism in \eqref{d.-e.}.

To prove the above claim, note that, since
$\Psi_U^* P^{\overline{Q}}$ is a reduction of
structure group of $\Psi^*_U
P^\PGL$, it is given by 
a section of $\Psi^*_U
P^\PGL(\PGL(N,\CC)/{\overline{Q}})$. But
$$
\Psi^*_U P^\PGL(\PGL(N,\CC)/{\overline{Q}})\, =\,
\Psi^*_U P^\GL(\GL(N,\CC)/Q)\, .
$$
Hence such a section
gives a reduction of structure group to
$Q\, \subset\, \GL(N,\CC)$ of the principal
$\GL(N,\CC)$--bundle associated to $\SW$. Since
this is equivalent to giving a subbundle
$\SH\,\inj \,\SW$, the above claim, that
the pullback
$\Psi_U^* P^{\overline{Q}}$ of the reduction
in \eqref{de.-re.} is given by a subbundle of $\SW$,
is proved.

If $\chi$ is a dominant character of $\overline{Q}$, then it is easy
to check that 
\begin{equation}
\label{eq:kp}
\deg(\Psi^*_U P^{\overline{Q}}(\chi))= k' (\rk(\SW)\deg (\SH)-
\deg(\SW)\rk(\SH))
\end{equation}
for some positive number $k'$. Indeed, this follows from
the earlier remark that any nontrivial dominant character
of $\overline{Q}$ is a positive multiple of the character
defined by the natural action of $\overline{Q}$ on the
line $\bigwedge^{\text{top}} \text{Hom}(\CC^N/A\, ,A)$.

In view of (\ref{eq:kp}), (\ref{eq:k}) and (\ref{stability}),
to prove the theorem it is enough to check that $\SH$ does not
contradict the stability of $\SW$, meaning
\begin{equation}
  \label{eq:positive}
 \rk(\SW)\deg (\SH)-\deg(\SW)\rk(\SH)<0 \; .
\end{equation}

If $r$ is even, note that
since $S_x$ is the image of $F(-x)_x$ in $E_x$ in diagram
\eqref{com}, it follows that $s\notin H^0(F(-x))$
(see \eqref{de.-ell}).

If $r$ is odd, then the vector space $S_x$ is the image
of $F'_x$ in $E_x$ in diagram \eqref{com'}, and it
follows that $s\notin H^0(F')$.

In both cases,
using this and \eqref{globalcom} (respectively, \eqref{globalcom'})
for the even (respectively, odd) case we conclude that the composition
\begin{equation}
\label{composition}
\SH\,\inj\, \SW\,\too\, \cS
\end{equation}
has nonzero image.

To unify the notation for the even and odd cases, denote
$$F_0\, :=\, F(-x)$$
if $r=2n$, and
$$
F_0\,:=\,F'
$$
if $r=2n+1$. Consider the commutative 
diagram
\begin{equation}
  \label{eq:com1}
\xymatrix{
{0} \ar[r] & {H^0(F_0)\otimes \SO_{Y_U}} \ar[r] & {\SW} \ar[r] & {\cS}\ar[r]
& 0\\
{0} \ar[r] & {\SH'} \ar[r] \ar@{^{(}->}[u] & 
{\SH} \ar[r] \ar@{^{(}->}[u] & {\SH''}\ar[r] \ar@{^{(}->}[u]
& 0\\
}
\end{equation}
We have seen that $\SH''\neq 0$. 
The vector bundle $\cS$ is stable of degree $-2$
(Proposition \ref{stables}). We have $\deg \SH''\leq -1$ 
because $\SH''$ is
a subsheaf of a stable vector bundle of negative degree,
and also $\deg \SH'\leq 0$ as it is a subsheaf of a semistable
vector bundle of degree zero; see \eqref{eq:com1}.
If $\deg \SH''\leq -2$, then
$\deg \SH\leq -2$, and hence (\ref{eq:positive}) holds.

Therefore, for the rest of the proof we consider 
the case $\deg \SH''=-1$.

Assume that (\ref{eq:positive}) does not hold, in other words,
assume that
\begin{equation}
  \label{eq:assume}
  \rk{\SH}\geq \frac{h^0(E)}{2}\, .
\end{equation}
We have $\rk \SH''<\rk \cS$, because if we had equality
we should have $$\deg \SH''\leq \deg \cS=-2\, ,$$ but we
are now in the case $\deg \SH''=-1$. Hence the stability
condition of $\cS$ implies that
\begin{equation}
  \label{eq:f''}
  \rk \SH'' < \frac{n}{2}\, .
\end{equation}

By our choice of  $(E,\varphi)\in \SM^0_L$, the
coherent sheaf $\SH$ is locally free, and
the induced homomorphism $\SH_{(E,\varphi)}\too \SW_{(E,\varphi)}$ is 
injective. Let 
$V'=\SH_{(E,\varphi)}\subset \SW_{(E,\varphi)}=H^0(E)$.
By Lemma \ref{sectionalstability}, Lemma \ref{Esemistable} and 
(\ref{eq:assume}),
\begin{equation}
  \label{eq:2}
\rk E^{V'} \geq \rk E \,\frac{\dim V'}{h^0(E)} =
r \,\frac{\rk \SH}{h^0(E)} \geq \frac{r}{2}\, .
\end{equation}
By our choice of $x\in X$,  the induced homomorphism
$E^{V'}_x\too E^{}_x$ is injective (see (\ref{eq:inj})). 
Restricting the commutative
diagram (\ref{eq:com1}) to the fiber over $(E,\varphi)$,
we obtain the following commutative diagram 
$$
\xymatrix{
 & & {E_x}\ar@{->>}^{p}[rd] & \\
{H^0(F_0)}\ar[r] & {H^0(E)}\ar[rr]\ar[ru]^{e} & & {S_x}  \ar[r] & {0}\\ 
 & {V'}\ar[rr]\ar[u] & & {\SH''_{(E,\varphi)}}  \ar[u]\ar[r] & {0}\\ 
}
$$
where $e$ is the evaluation morphism and 
the projection $p$ fits in an exact sequence
$$
0 \too N_x \too E_x \stackrel{p}\too S_x \too 0\, .
$$
It follows that
$$
\SH''_{(E,\varphi)} = \textup{Im} (V'\to S_x) = 
\textup{Im}(V'\inj H^0(E) \to E_x \to S_x)
$$
$$
= \textup{Im}(E^{V'}_x\inj E_x \to S_x) \cong
\frac{E^{V'}_x}{E^{V'}_x\cap N_x}\, .
$$

It is easy to check that
$$
\dim \frac{E^{V'}_x}{E^{V'}_x\cap N_x} \leq
\dim S_x \; ,
$$
with equality holding if $N_x$ is a general subspace of $E_x$.

On the other hand, $\dim S_x=n$, and therefore,
for a general $N_x\subset E_x$, we have $\rk \SH''=n$,
which contradicts (\ref{eq:f''}). This implies that
the assumption in (\ref{eq:assume}) is false, meaning
(\ref{eq:positive}) holds, and the theorem if proved.
\end{proof}

\begin{lemma}
\label{sectionalstability}
Let $E$ be a semistable vector bundle on $X$ of degree $d$ 
and rank $r$. 
If $d > (2g-2)r +r$, then for all 
subspaces $V'\subset H^0(E)$,
$$
r \dim V' - \rk E^{V'} h^0(E) \leq 0\, ,
$$
where $E^{V'}\subset E$ is the subsheaf generated
by $V'\subset H^0(E)$.
\end{lemma}

\begin{proof}
Let $E'$ be a semistable vector bundle of degree $d'$ and
rank $r'$. If $d'>(2g-2)r'$, then
\begin{equation}
  \label{eq:case1}
  h^0(E')=d'-r'(g-1)
\end{equation}
because $h^1(E')=h^0(K_X\otimes E'^\vee)$ (Serre
duality), and this is zero
since $K_X\otimes E'^\vee$ is a semistable vector bundle of
negative degree. On the other hand, 
if $0\leq d' \leq (2g-2)r'$, then by Clifford's 
theorem (see, for instance, \cite[Theorem 2.1]{BGN}),
\begin{equation}
  \label{eq:case2}
  h^0(E')\leq \frac{d'}{2}+r'\, .
\end{equation}
Therefore, if $h^0(E')>(g-1)r'+r'$, then we must be in
the first case, and hence
\begin{equation}
  \label{eq:high}
  d'=h^0(E')+r'(g-1) > 2r'(g-1)+r'
\end{equation}
if $h^0(E')>(g-1)r'+r'$.

Let $E^i$, $1\,\leq\, i\,\leq\,\ell$, be the successive
quotients of the Harder--Narasimhan
filtration of $E^{V'}$. We have
$$
\dim V'\leq h^0(E^{V'}) \leq \sum h^0(E^i)
$$
for all $i$. Denote $r_i=\rk E^i$. If
$h^0(E^i)\leq (g-1)r_i+r_i$, then applying
\eqref{eq:case1} to $E$, we have
$$
r\cdot h^0(E^i) - r_i\cdot h^0(E) \,<\,
r (r_i(g-1)+r_i) - r_i (r (g-1)+r)\leq 0\, .
$$
On the other hand, if $h^0(E^i)> (g-1)r_i+r_i$,
then applying \eqref{eq:high} and \eqref{eq:case1} 
to $E^i$ and $E$ we have
$$
r\cdot h^0(E^i) - r_i\cdot h^0(E)  = 
r (d^i-r_i(g-1))-r_i(d-r (g-1)) \leq 0
$$
by the stability condition of $E$.
Therefore
$$
r \dim V'-\rk E^{V^i}h^0(E) \leq
\sum \big(r h^0(E^i) - r_i h^0(E)\big) \leq
0\, .
$$
This completes the proof of the lemma.
\end{proof}

For a fixed holomorphic line bundle $\xi$ on $X$, 
sending any orthogonal bundle 
$\varphi:E\otimes E\too L$
to the orthogonal bundle
$(E\otimes\xi)\otimes (E\otimes\xi)\too L\otimes\xi^{\otimes 2}$ 
given by $\varphi$
we obtain an isomorphism between the corresponding moduli
spaces. The pull-back of the projective Picard bundle
under this isomorphism 
is the \emph{$\xi$-twisted projective Picard bundle}, 
namely the projective bundle whose fiber over
a point corresponding to $(E,\varphi)$ is canonically
isomorphic to $\PP(H^0(E\otimes\xi))$. 
Using this isomorphism, we have the following corollary.

\begin{corollary}
Assume that
$g(X)\,>\,n$ if $r=2n+1$ is odd, and $g(X)\,>\,n+1$ if
$r=2n$ is even. Let $\xi$ be a holomorphic line bundle of
degree $m$. Assume also $d+rm>(2g-2)r+r$.
Then, the $\xi$-twisted projectivized Picard bundle $P^\PGL$
over the moduli space $\SM^0_L$
is stable with respect to the unique polarization
of $\SM^0_L$.
\end{corollary}

In \cite{BG}, we introduced Hecke transform for symplectic
bundles. In the symplectic case we started with a $(0,n)$--stable
symplectic bundle, performed a Hecke transform with an isotropic 
subspace of dimension $n$, and obtained a stable symplectic bundle.

In the odd orthogonal case, i.e., when the group is $\GO(2n+1,\CC)$, 
we cannot choose a subspace of half dimension.
To obtain an orthogonal bundle in this case, we have to
start with a singular bilinear form $(F,F\otimes F\too M)$ 
(Proposition \ref{prop:odd1}), 
which does not come from a principal bundle. 
The naive approach would be to start with a $(0,n)$--stable
bundle as in Proposition \ref{prop:odd1} to obtain an
orthogonal bundle. The difficulty is that not much
is known about the properties of the moduli space of
bilinear forms with singularities like this
(neither smoothness nor dimension is known), and in
particular, we do not know if the set of $(0,n)$--stable
bundles is dense.

The strategy is then to perform two Hecke transforms
instead of just one: we start with an $(n,n)$--stable orthogonal
bundle. A Hecke transform with respect to a subspace of
dimension $n$ will produce a singular (in the odd case)
bilinear bundle
$(F,F\otimes F\too M$, which is 
$(0,n)$-stable (Proposition \ref{prop:odd2}), and then we perform
a second Hecke transform (see Proposition \ref{prop:odd1}) 
to get a stable orthogonal bundle. Therefore, 
we have to prove that the set of $(n,n)$--stable orthogonal bundles is
dense (Proposition \ref{nnstable}).

Another difficulty in the orthogonal case
is that the determinant of the universal vector subbundle $\cS$
on the Grassmannian parameterizing rank $n$ isotropic subspaces in
$\CC^{2n}$ (endowed with the standard orthogonal form) is
not a primitive element of the Picard group
(see Proposition \ref{stables}). 
Because of this, the proof of
Theorem \ref{thm.} is longer than in the symplectic
case, where the determinant of the analogous
universal bundle is a generator of the Picard group.

\end{document}